\documentclass[11pt,hidelinks]{amsart}

\usepackage[margin=1.28in]{geometry}
\usepackage{amsmath,amssymb,amsthm,mathtools}
\usepackage{enumitem}
\usepackage{xcolor}
\usepackage{tikz-cd}
\usepackage{hyperref}
\usepackage{orcidlink,tikz-cd}
\usepackage{mathrsfs}

\definecolor{darkgoldenrod}{rgb}{0.72,0.53,0.04}
\definecolor{goldmetallic}{rgb}{0.83,0.69,0.22}
\hypersetup{
  colorlinks=true,
  linkcolor=darkgoldenrod,
  urlcolor=goldmetallic,
  citecolor=darkgoldenrod
}

\usepackage{amsmath,amssymb,amsthm,mathtools}
\usepackage{enumitem}
\usepackage{microtype}
\usepackage{hyperref}
\usepackage[backend=biber,style=numeric,sorting=nyt,maxbibnames=99]{biblatex}
\newtheorem{theorem}{Theorem}[section]
\newtheorem{proposition}[theorem]{Proposition}
\newtheorem{lemma}[theorem]{Lemma}
\newtheorem{corollary}[theorem]{Corollary}
\newtheorem{remark}[theorem]{Remark}

\newcommand{\Q}{\mathbf Q}
\newcommand{\Z}{\mathbf Z}
\newcommand{\F}{\mathbf F}

\newcommand{\Gal}{\operatorname{Gal}}
\newcommand{\GL}{\operatorname{GL}}
\newcommand{\PGL}{\operatorname{PGL}}
\newcommand{\GSp}{\operatorname{GSp}}
\newcommand{\Sp}{\operatorname{Sp}}
\newcommand{\PGSp}{\operatorname{PGSp}}
\newcommand{\ind}{\operatorname{ind}}
\newcommand{\Disc}{\operatorname{Disc}}
\newcommand{\Jac}{\operatorname{Jac}}

\newcommand{\PP}{\mathbf P}

\title[Counting symplectic Galois number fields]{Counting number fields with symplectic Galois group}
\author[A.~Ray]{Anwesh Ray\, \orcidlink{0000-0001-6946-1559}}
\address{Chennai Mathematical Institute, H1, SIPCOT IT Park, Siruseri, Kelambakkam 603103, Tamil Nadu, India.}
\email{anwesh@cmi.ac.in}

\subjclass[2020]{11R32, 11R29, 11G30, 14H40}
\keywords{number fields, symplectic groups, Jacobians, Galois representations, discriminants, hyperelliptic curves}

\begin{document}

\begin{abstract}
Let $n\geq 1$ and let $\ell$ be an odd prime.  Let
$G=\mathrm{PGSp}_{2n}(\F_\ell)$ or $\mathrm{GSp}_{2n}(\F_\ell)$, and fix a faithful
transitive permutation representation
$\pi:G\longrightarrow S_d$. We study degree-$d$ number fields whose
Galois closures have Galois group $G$ and whose associated permutation
representation is $\pi$. For $\sigma\in S_d$, write
$\operatorname{ind}(\sigma)$ for its permutation index, namely
$\operatorname{ind}(\sigma)
=
d-\#\{\text{orbits of $\sigma$ on $\{1,\ldots,d\}$}\}$.
If $\tau$ denotes a symplectic transvection, or its image in the
projective symplectic group, we prove that the number of such fields
with absolute discriminant at most $X$ is bounded below by a constant
multiple of
$X^{1/(2n\,\operatorname{ind}(\pi(\tau)))}$. For the natural vector and projective actions, the exponents we obtain are asymptotically $1/(4n)$ of those predicted by the weak form of Malle's conjecture as $\ell\to\infty$. The fields are constructed from the mod-$\ell$ Galois representations
attached to the Jacobians of a one-parameter family of hyperelliptic
curves. The proof combines large symplectic monodromy for this family with a
squarefree sieve.
\end{abstract}
\maketitle

\section{Introduction}\label{sec:introduction}
We place our results in the context of Malle's conjecture. Let
$\pi:G\rightarrow S_d$ be a faithful transitive permutation
representation of a finite group $G$, and set
\[
a_\pi(G)
=
\min_{1\neq g\in G}\operatorname{ind}(\pi(g)),
\]
where $\operatorname{ind}(\sigma)$ denotes the permutation index of
$\sigma\in S_d$. Let $N_\pi(G;X)$ denote the number of degree-$d$
number fields with absolute discriminant at most $X$, whose Galois
closures have Galois group $G$ and whose associated permutation
representation is equivalent to $\pi$. Then the strong form of
Malle's conjecture predicts that
\[N_\pi(G;X)\sim c_{\pi,G}X^{1/a_\pi(G)}
(\log X)^{b_{\pi,G}-1}\]
for certain constants $c_{\pi,G}>0$ and $b_{\pi,G}$; see \cite{Malle2002,Malle2004}. The precise
prediction for the logarithmic exponent is known to require
modification in general, following the counterexamples of Kl\"uners
\cite{Kluners2005}; see also \cite{Turkelli2015}. The weak form of Malle's conjecture states that for every $\epsilon>0$, one has
\[
X^{1/a_\pi(G)-\epsilon}
\ll_{\pi,G,\epsilon}
N_\pi(G;X)
\ll_{\pi,G,\epsilon}
X^{1/a_\pi(G)+\epsilon}.
\]

In the setting of this paper, where
$G=\GSp_{2n}(\F_\ell)$ or $\PGSp_{2n}(\F_\ell)$, our lower bounds are
governed by the permutation index of a symplectic transvection.  If
$\tau$ denotes such a transvection, or its image in the projective
group, we prove
\[
N_\pi(G;X)
\gg
X^{1/(2n\,\operatorname{ind}(\pi(\tau)))}.
\]

First consider the case when $n=1$, i.e., for the groups $\GL_2(\F_\ell)$ and $\PGL_2(\F_\ell)$. In \cite{Ray2025}, a two-parameter family of elliptic curves is combined with a large-sieve estimate controlling the multiplicity with which a residual representation occurs. Cristante \cite{Cristante2025} instead uses a one-parameter family with squarefree discriminant. The purpose of the present paper is to study generalizations for symplectic representations arising from Jacobians of hyperelliptic curves. There is a large supply of Jacobians with maximal residual image. Density-one maximal-image results in arithmetic families were established in \cite{LandesmanEtAl2019,LandesmanEtAl2020}, and explicit constructions with surjective residual representations are given in \cite{AnniDokchitser2020}. For the counting problem considered here, however, it is useful to work with a one-parameter family in which the varying primes of bad reduction are detected by the prime divisors of a single polynomial in the parameter. We shall return to the construction after stating the main results.

\subsection{Main results}

Fix $n\geq 1$ and an odd prime $\ell$, and let $V$ be a $2n$-dimensional symplectic $\F_\ell$-space. Put $G^+=\GSp(V)$ and $G=\PGSp(V)$. A symplectic transvection is a nontrivial element $\tau\in\Sp(V)$ for which $\tau-1$ has rank one. Although for odd $\ell$ there are two $\Sp(V)$-conjugacy classes of transvections, all nontrivial symplectic transvections are conjugate in $\GSp(V)$ (see Lemma~\ref{lem:GSp-conjugacy}). Consequently the index occurring in the statements below is independent of the transvection chosen.

The projective result is the principal theorem of the paper.

\begin{theorem}\label{thm:PGSp-main}
Let $n\geq 1$ and let $\ell$ be an odd prime. Let
$\pi:\PGSp_{2n}(\F_\ell)\hookrightarrow S_d$ be a faithful transitive permutation representation, and let $\overline\tau$ denote the projective image of a symplectic transvection. Then
\[
 N_\pi\bigl(\PGSp_{2n}(\F_\ell);X\bigr)
 \gg_{n,\ell,\pi}
 X^{1/(2n\,\ind(\pi(\overline\tau)))}.
\]
\end{theorem}

We also obtain the corresponding linear statement for an arbitrary faithful transitive permutation representation of $\GSp_{2n}(\F_\ell)$.

\begin{theorem}\label{thm:GSp-main}
Let $n\geq 1$ and let $\ell$ be an odd prime. Let
$\pi:\GSp_{2n}(\F_\ell)\hookrightarrow S_d$ be a faithful transitive permutation representation, and let $\tau$ be a symplectic transvection. Then
\[
 N_\pi\bigl(\GSp_{2n}(\F_\ell);X\bigr)
 \gg_{n,\ell,\pi}
 X^{1/(2n\,\ind(\pi(\tau)))}.
\]
\end{theorem}

The two most immediate examples come from the natural vector and projective actions.

\begin{corollary}\label{cor:natural-actions}
For the natural action of $\GSp_{2n}(\F_\ell)$ on $V\setminus\{0\}$, which has degree $\ell^{2n}-1$, one has
\[
 N_\pi\bigl(\GSp_{2n}(\F_\ell);X\bigr)
 \gg_{n,\ell}
 X^{1/(2n\ell^{2n-2}(\ell-1)^2)}.
\]
For the natural projective action of $\PGSp_{2n}(\F_\ell)$ on $\PP(V)$, which has degree $(\ell^{2n}-1)/(\ell-1)$, one has
\[
 N_\pi\bigl(\PGSp_{2n}(\F_\ell);X\bigr)
 \gg_{n,\ell}
 X^{1/(2n\ell^{2n-2}(\ell-1))}.
\]
\end{corollary}

When $n=1$, one has $\GSp_2(\F_\ell)=\GL_2(\F_\ell)$ and $\PGSp_2(\F_\ell)=\PGL_2(\F_\ell)$. The projective action has degree $\ell+1$, and Corollary~\ref{cor:natural-actions} gives the exponent $1/(2(\ell-1))$. Thus Theorem~\ref{thm:PGSp-main} recovers, in dimension two, the exponent furnished by the one-parameter $\PGL_2$ construction of \cite{Cristante2025}.

For these two permutation representations, the comparison with the
weak form of Malle's conjecture can be made explicit.  For $n\geq 2$,
the Malle invariants for the natural vector and projective actions are,
respectively,
\begin{align*}
a_\pi\bigl(\GSp_{2n}(\F_\ell)\bigr)
&=\frac{\ell^{2n-2}(\ell^2-1)}{2},\\
a_\pi\bigl(\PGSp_{2n}(\F_\ell)\bigr)
&=\frac{(\ell+1)(\ell^{2n-2}-1)}{2}.
\end{align*}
In both cases the minimum is attained by an involution. In the natural
vector action, when $\ell=3$, a transvection also attains the minimum;
for $\ell>3$, the involution has strictly smaller index. Thus weak Malle
predicts the respective exponents
\[
\frac{2}{\ell^{2n-2}(\ell^2-1)}
\qquad\text{and}\qquad
\frac{2}{(\ell+1)(\ell^{2n-2}-1)}.
\]
Consequently, the exponents in Corollary~\ref{cor:natural-actions} are
respectively
\[
\frac{\ell+1}{4n(\ell-1)}
\qquad\text{and}\qquad
\frac{(\ell+1)(\ell^{2n-2}-1)}
{4n\ell^{2n-2}(\ell-1)}
\]
times those predicted by weak Malle.  In particular, for fixed
$n\geq2$ and $\ell\to\infty$, both ratios tend to $1/(4n)$.
When $n=1$, the corresponding Malle invariants are
$\ell(\ell-1)/2$ for the natural $\GL_2(\F_\ell)$ action on
$\F_\ell^2\setminus\{0\}$ and $(\ell-1)/2$ for the natural
$\PGL_2(\F_\ell)$ action on $\PP^1(\F_\ell)$.

\subsection{Methodology}

We use the one-parameter hyperelliptic pencil $C_T$ given by \[y^2=(x-T)f(x),\] where $m=2\ell(2n)!$, $a_i=im$ for $1\leq i\leq 2n$, and $f(x)=\prod_{i=1}^{2n}(x-a_i)$.  Set $T=1+mu$, and \[P(u):=f(1+mu)=\prod_{i=1}^{2n}(1+m(u-i)).\]
The generic geometric image on $\ell$-torsion is $\Sp_{2n}(\F_\ell)$ and the arithmetic image is $\GSp_{2n}(\F_\ell)$.  Quantitative Hilbert irreducibility shows that the specializations with smaller image have density zero. If $p\nmid m$ divides $P(u)$ exactly once, then the mod-$\ell$ inertia at $p$ is generated by a transvection. A squarefree sieve supplies a positive proportion of parameters for which all varying bad primes are of this form. Their splitting fields are distinct since the sets of primes that ramify in these fields are distinct. The tame discriminant formula gives $|\Disc K_u|\ll P(u)^{I_\pi}\ll Y^{2nI_\pi}$, where $I_\pi$ is the permutation index of a transvection. This yields Theorems~\ref{thm:PGSp-main} and~\ref{thm:GSp-main}.

\subsection{Organization}

In Section~\ref{sec:preliminaries} we recall the tame discriminant formula, prove the uniform bound needed at the fixed primes, and delineate the conjugacy properties of symplectic transvections. The generic monodromy and quantitative specialization estimates are established in Section~\ref{sec:monodromy}. The local analysis at the varying primes is carried out in Section~\ref{sec:local}, and the squarefree sieve is applied in Section~\ref{sec:sieve}. The main theorems are proved in Section~\ref{sec:proofs}. In Section~\ref{sec:actions} we compute the transvection index for natural permutation actions and prove Corollary \ref{cor:natural-actions}.

\section{Preliminaries}\label{sec:preliminaries}

\subsection{Permutation indices and discriminants}

Let $\Omega$ be a finite set of cardinality $d$, and let $\sigma\in S_\Omega$. We define the permutation index by
$\ind(\sigma)=d-\#(\operatorname{Orbit}\sigma)$. Equivalently, if the cycle lengths of $\sigma$ are $e_1,\ldots,e_r$, then $\ind(\sigma)=\sum_{j=1}^r(e_j-1)$.

Suppose that $L/\Q$ is finite Galois with group $G$, let $H\leq G$, and set $K=L^H$. The natural action of $G$ on $G/H$ is the permutation representation associated with the embeddings of $K$ into a fixed algebraic closure. We shall repeatedly use the following consequence of the conductor--discriminant formula.

\begin{lemma}[Tame discriminant formula]\label{lem:tame-disc}
Let $p$ be a prime which is tamely ramified in $L/\Q$, and suppose that its inertia group is generated by $\sigma\in G$. If $\pi$ denotes the action of $G$ on $G/H$, then $v_p(\Disc K)=\ind(\pi(\sigma))$.
\end{lemma}

For further details, see \cite[Chapter VI, \S2]{SerreLocalFields}.

\begin{lemma}\label{lem:fixed-prime}
Let $K/\Q$ be a number field of degree $d$. For every prime $p$, one has \[v_p(\Disc K)\leq d-1+d\lfloor\log_p d\rfloor.\] In particular, if $S$ is a fixed finite set of primes, then there is a constant $C_{d,S}$ such that the $S$-part of $|\Disc K|$ is at most $C_{d,S}$ for every degree $d$ field $K$.
\end{lemma}

\begin{proof}
Write $K\otimes_\Q\Q_p=\prod_{v\mid p}K_v$, and let $e_v$ and $f_v$ denote the ramification index and residue degree of
$K_v/\Q_p$.  If $\delta_v$ is the exponent of the different of
$K_v/\Q_p$, then
\[\delta_v\leq e_v-1+e_vv_p(e_v);\] see
\cite[Chapter III, \S6]{SerreLocalFields}.  Consequently,
\[
 v_p(\Disc K)=\sum_{v\mid p}f_v\delta_v
 \leq
 \sum_{v\mid p}f_v(e_v-1)
 +\sum_{v\mid p}f_ve_vv_p(e_v).
\]
The first sum is
$\sum_{v\mid p}e_vf_v-\sum_{v\mid p}f_v$ and is therefore at most
$d-1$, since $\sum_{v\mid p}e_vf_v=d$ and at least one prime of $K$
lies above $p$.  For the second sum, each $e_v$ is at most $d$, so
$v_p(e_v)\leq\lfloor\log_p d\rfloor$.  It follows that $\sum_{v\mid p}f_ve_vv_p(e_v)\leq d\lfloor\log_p d\rfloor$. This proves the claimed bound.  If $S$ is fixed, multiplying the
corresponding powers of the finitely many primes in $S$ gives a constant
$C_{d,S}$ depending only on $d$ and $S$.
\end{proof}

\subsection{Symplectic transvections}

We fix a nondegenerate alternating form $\langle\ ,\ \rangle$ on a $2n$-dimensional $\F_\ell$-space $V$. If $v\in V\setminus\{0\}$ and $a\in\F_\ell^\times$, we set
\[T_{v,a}(x):=x+a\langle x,v\rangle v.\] This is a symplectic transvection. Observe that $T_{cv,a}=T_{v,ac^2}$ for $c\in\F_\ell^\times$.

Conjugacy of transvections in $\Sp(V)$ and in $\GSp(V)$ behaves slightly differently: for odd $\ell$ there may be two $\Sp(V)$-classes, whereas there is only one $\GSp(V)$-class.

\begin{lemma}\label{lem:GSp-conjugacy}
All nontrivial symplectic transvections are conjugate in $\GSp(V)$. Consequently their images are conjugate in $\PGSp(V)$. The image of a transvection in $\PGSp(V)$ has order $\ell$.
\end{lemma}

\begin{proof}
Let $T_{v,a}$ and $T_{w,b}$ be two transvections. Suppose that $g\in\GSp(V)$ has multiplier $\mu$, so $\langle gx,gy\rangle=\mu\langle x,y\rangle$. A direct calculation gives
$gT_{v,a}g^{-1}=T_{gv,a\mu^{-1}}$.

Choose symplectic bases
\[
v,v^*,v_2,v_2^*,\ldots,v_n,v_n^*
\qquad\text{and}\qquad
w,w^*,w_2,w_2^*,\ldots,w_n,w_n^*
\]
of $V$, normalized so that
$\langle v,v^*\rangle=\langle w,w^*\rangle=1$.
For $c,\mu\in\F_\ell^\times$, define $g\in\GSp(V)$ by
\[
g(v)=cw\quad\text{and}\quad g(v^*)=\mu c^{-1}w^*,
\]
and, for $j\geq2$, by
\[
g(v_j)=w_j\quad\text{and}\quad g(v_j^*)=\mu w_j^*.
\]
Then $g$ has multiplier $\mu$. Using
$gT_{v,a}g^{-1}=T_{gv,a\mu^{-1}}$, we obtain
$gT_{v,a}g^{-1}=T_{w,a\mu^{-1}c^2}$.
Taking $\mu=ac^2/b$ gives $gT_{v,a}g^{-1}=T_{w,b}$.
Thus any two nontrivial symplectic transvections are conjugate in
$\GSp(V)$.

Conjugacy of the projective images follows immediately. Finally, a transvection has order $\ell$. If a nontrivial power of it were scalar, that scalar would be a unipotent scalar and hence equal to $1$, forcing the power itself to be the identity. Thus its image modulo the scalar centre still has order $\ell$.
\end{proof}

\subsection{A general counting result}
The result below will be applied first to $\GSp$ and then to $\PGSp$.

\begin{proposition}\label{prop:counting-principle}
Let $G$ be a finite group, let $\tau\in G$ have prime order $\ell$, and let $P\in\Z[T]$ have degree $r\geq1$. Suppose that $S$ is a fixed finite set of primes containing $\ell$ and that, for all sufficiently large $Y$, there are $\gg Y$ integers $u$ with $1\leq u\leq Y$ with the following properties.
\begin{enumerate}[label=\textup{(\roman*)}]
\item $P(u)$ is positive and squarefree, no prime in $S$ divides $P(u)$, and the values $P(u)$ occurring in the chosen set are pairwise distinct;
\item there is a Galois extension $L_u/\Q$ with $\Gal(L_u/\Q)\simeq G$;
\item for every prime $p\notin S$, the extension $L_u/\Q$ is ramified at $p$ if and only if $p\mid P(u)$, and in that case the inertia group is cyclic of order $\ell$ and is generated by an element conjugate to $\tau$.
\end{enumerate}
Then, for every faithful transitive permutation representation $\pi$ of $G$ in $S_d$,
\[
 N_\pi(G;X)\gg_{G,\pi,P,S}X^{1/(r\,\ind(\pi(\tau)))}.
\]
\end{proposition}

\begin{proof}
Let $H$ be the stabilizer of a point of the $G$-set underlying $\pi$, and
put $K_u=L_u^H$.  Since $\pi$ is faithful, the core
$\bigcap_{g\in G}gHg^{-1}$ is trivial.  Hence the normal closure of
$K_u/\Q$ inside $L_u$ is all of $L_u$, and the action of
$\Gal(L_u/\Q)$ on the embeddings of $K_u$ is permutation-isomorphic to
$\pi$.  In particular, every $K_u$ is one of the fields counted by
$N_\pi(G;X)$ once its discriminant is at most $X$.

Set $I_\pi=\ind(\pi(\tau))$.  Let $p\notin S$ divide $P(u)$.  By
hypothesis (iii), inertia at $p$ is generated by an element conjugate to
$\tau$.  Since $p\neq\ell$, this ramification is tame, and
Lemma~\ref{lem:tame-disc} gives $v_p(\Disc K_u)=I_\pi$. The number $I_\pi$ is positive. If it were zero, $\pi(\tau)$ would be
the identity, contradicting the faithfulness of $\pi$.  If
$p\notin S$ does not divide $P(u)$, hypothesis (iii) says that $p$ is
unramified.  Since $P(u)$ is squarefree and is prime to the primes in
$S$, the contribution to the discriminant from primes outside $S$ is
therefore exactly $P(u)^{I_\pi}$. By Lemma~\ref{lem:fixed-prime}, the contribution from the primes in $S$ to the
discriminant of any degree-$d$ field is bounded by a constant depending
only on $d$ and $S$.  We obtain $|\Disc K_u|\ll_{d,S}P(u)^{I_\pi}$. As $P$ has degree $r$, there is a constant depending only on $P$ such
that $|P(u)|\ll_PY^r$ for $1\leq u\leq Y$.  Hence
$|\Disc K_u|\ll_{P,d,S}Y^{rI_\pi}$.

It remains to count distinct fields.  At every prime outside $S$ which
divides $P(u)$, the inertia generator has nontrivial image under $\pi$,
so the degree-$d$ field $K_u$ is ramified there.  Thus its ramification
support outside $S$ is exactly the prime support of $P(u)$.  Two
positive squarefree integers have the same prime support if and only if
they are equal.  Hypothesis (i) therefore implies that distinct chosen
values of $u$ give fields with different ramification supports, and
hence nonisomorphic fields.  We have consequently produced $\gg Y$
distinct fields of discriminant $\ll Y^{rI_\pi}$.  Taking
$Y=cX^{1/(rI_\pi)}$ with $c>0$ sufficiently small proves the result.
\end{proof}

\section{The hyperelliptic pencil and its generic monodromy}\label{sec:monodromy}

Fix $n\geq1$ and an odd prime $\ell$. Set
$m=2\ell(2n)!$, put $a_i=im$ for $1\leq i\leq2n$, and let $f(x)=\prod_{i=1}^{2n}(x-a_i)$. Let $\mathcal U=\mathbf A^1_\Q\setminus\{a_1,\ldots,a_{2n}\}$. Over $\mathcal U$ we take the smooth proper model of the affine family $\mathcal C_T$ given by $y^2=(x-T)f(x)$. Since the polynomial on the right has degree $2n+1$, every fibre over $\mathcal U$ is a smooth hyperelliptic curve of genus $n$. Let $\mathcal J\to\mathcal U$ denote the relative Jacobian. The principal polarization and the Weil pairing give a representation \[\rho_{\ell}: \pi_1(\mathcal U)\rightarrow \GSp_{2n}(\F_\ell).\] We write $G^{\mathrm{geom}}_\ell$ for the image of the geometric fundamental group and $G^{\mathrm{arith}}_\ell$ for the full arithmetic image.

\begin{lemma}\label{lem:fixed-primes}
Every prime divisor of $2\ell\Disc(f)$ divides $m$. If $T=1+mu$ with $u\in\Z$ and
$P(u)=f(1+mu)$, then $\gcd(P(u),m)=1$.
\end{lemma}

\begin{proof}
Up to sign, $\Disc(f)$ is the product of $(a_i-a_j)^2$ over $i<j$. Since $a_i-a_j=m(i-j)$ and every prime divisor of $i-j$ is at most $2n$, every prime dividing $\Disc(f)$ divides $m$. The primes $2$ and $\ell$ also divide $m$ by definition.

For the second assertion, if $p\mid m$, then each factor $1+m(u-i)$ is congruent to $1$ modulo $p$. Thus no prime divisor of $m$ can divide $P(u)$.
\end{proof}

We next prove the generic monodromy statement in the exact form needed later. Hall's theorem applies to the one-parameter pencil used here. We first compare the geometric monodromy in characteristic zero with the
monodromy obtained after reduction at a suitable auxiliary prime.

\begin{lemma}\label{lem:special-fibre-monodromy}
Let $\ell$ be an odd prime and choose a rational prime
$q>\max\{m,|\GSp_{2n}(\F_\ell)|\}$.  Let $\overline f$ denote the
reduction of $f$ modulo $q$.  Then $\overline f$ is squarefree of
degree $2n$, and the geometric monodromy on the $\ell$-torsion of the
Jacobian of
\[
y^2=(x-T)\overline f(x)
\]
over $\overline{\F}_q(T)$ is $\Sp_{2n}(\F_\ell)$.
\end{lemma}

\begin{proof}
Since $q>m$, we have $q\nmid m$.  In particular $q$ is odd,
$q\neq\ell$, and, by Lemma~\ref{lem:fixed-primes}, $q$ does not divide
$\Disc(f)$.  We can also see the last assertion directly.  If two of
the fixed branch points had the same reduction modulo $q$, say
$a_i\equiv a_j\pmod q$, then $q$ would divide
$a_i-a_j=m(i-j)$.  Since $q\nmid m$, this would imply $q\mid i-j$.
This is impossible because $0<|i-j|<2n<q$.  Thus the reductions of
$a_1,\ldots,a_{2n}$ are pairwise distinct, and $\overline f$ is monic
and squarefree of degree $2n$.

We may now apply Hall's theorem.  In the notation of
\cite[Theorem~4.1]{HallBig2008}, take the function field
$K=\F_q(T)$ and the hyperelliptic curve with affine equation
$y^2=(T-x)\overline f(x)$.  Hall proves, for every odd prime $\ell$
different from the characteristic, that after adjoining the
$\ell$th roots of unity the $\ell$-torsion extension is geometric
with Galois group $\Sp_{2n}(\F_\ell)$.  Equivalently, the image of the
geometric fundamental group on the $\ell$-torsion is the full
symplectic group.

Our equation has $x-T$ in place of $T-x$.  This does not change the
geometric monodromy.  Indeed, over $\overline{\F}_q$ we may choose a
square root of $-1$, and multiplication of the $y$-coordinate by this
square root gives an isomorphism between the two curves.  Their
geometric $\ell$-torsion local systems are therefore isomorphic.
Hence the geometric image for our reduced family is
$\Sp_{2n}(\F_\ell)$.
\end{proof}

We next compare the geometric monodromy of the characteristic-zero
family with that of its reduction modulo $q$.

\begin{lemma}\label{lem:specialization-monodromy}
With $q$ chosen as in Lemma~\ref{lem:special-fibre-monodromy}, the
geometric monodromy group of the characteristic-zero family is equal,
up to conjugacy, to the geometric monodromy group of its reduction
modulo $q$.  In particular,
$G^{\mathrm{geom}}_\ell=\Sp_{2n}(\F_\ell)$.
\end{lemma}

\begin{proof}
Consider the relative projective line
$\PP^1_{\Z_{(q)}}$ and remove the sections
$a_1,\ldots,a_{2n},\infty$.  We denote the resulting open subscheme by
$\mathscr U$.  Since the $a_i$ remain pairwise distinct modulo $q$,
these sections form a relative normal-crossings divisor in
$\PP^1_{\Z_{(q)}}$.

The hyperelliptic pencil extends over $\mathscr U$. The smooth
projective models of $y^2=(x-T)f(x)$ form a smooth proper family
$\mathscr C\rightarrow\mathscr U$.  Let
$\mathscr J\rightarrow\mathscr U$ be its relative Jacobian.  This is
an abelian scheme of relative dimension $n$.  Since $q\neq\ell$, the
group scheme $\mathscr J[\ell]$ is finite etale over $\mathscr U$.

Replace
$\Z_{(q)}$ by its strict henselization at $q$, and let
$\overline\eta$ and $\overline s$ denote geometric generic and
geometric special points.  Write
$\mathscr U_{\overline\eta}$ and $\mathscr U_{\overline s}$ for the
corresponding geometric fibres. For this smooth relative curve with relative normal-crossings boundary,
the specialization theorem for tame fundamental groups gives a
surjective homomorphism
\[
\operatorname{sp}\colon
\pi_1^{t}(\mathscr U_{\overline\eta})
\longrightarrow
\pi_1^{t}(\mathscr U_{\overline s})
\]
which induces an isomorphism on the maximal prime-to-$q$ quotients;
see \cite[Expos\'e XIII, \S2.10 and Corollaire~2.12]{SGA1}. In the present setting it is enough to work with the maximal prime-to-$q$ quotient.
Indeed, both geometric monodromy groups are subgroups of
$\GSp_{2n}(\F_\ell)$, and our choice of $q$ gives
$q>|\GSp_{2n}(\F_\ell)|$.  Their orders are therefore prime to $q$.
On the special fibre this also means that wild inertia, which is a
pro-$q$ group, acts trivially.  Hence the two $\ell$-torsion
representations factor through the maximal prime-to-$q$ quotients of
the corresponding tame fundamental groups. Since specialization is an isomorphism
on the prime-to-$q$ quotients through which they factor, their images
are the same, up to conjugacy.

By Lemma~\ref{lem:special-fibre-monodromy}, the image on the geometric
special fibre is $\Sp_{2n}(\F_\ell)$.  It follows that the
characteristic-zero geometric image is also
$\Sp_{2n}(\F_\ell)$.
\end{proof}

We can now determine the full arithmetic image.

\begin{proposition}\label{prop:generic-image}
For every odd prime $\ell$ one has
$G^{\mathrm{geom}}_\ell=\Sp_{2n}(\F_\ell)$ and
$G^{\mathrm{arith}}_\ell=\GSp_{2n}(\F_\ell)$.
\end{proposition}

\begin{proof}
The geometric assertion follows from
Lemma~\ref{lem:specialization-monodromy}.  We therefore only have to pass
from geometric to arithmetic monodromy.

The principal polarization on the relative Jacobian gives the Weil
pairing on $\mathcal J[\ell]$.  Compatibility of the Galois action
with this pairing shows that the multiplier of $\rho_\ell$ is the
mod-$\ell$ cyclotomic character.  On the geometric fundamental group
the roots of unity are constant, so the multiplier is trivial.  This
also explains directly why the geometric image is contained in
$\Sp_{2n}(\F_\ell)$.

Since $\mathcal U$ is geometrically connected, its arithmetic and
geometric fundamental groups fit into the usual exact sequence
\[
1\longrightarrow
\pi_1(\mathcal U_{\overline{\Q}})
\longrightarrow
\pi_1(\mathcal U)
\longrightarrow
G_\Q
\longrightarrow1.
\]
The image of the first term is
$G^{\mathrm{geom}}_\ell=\Sp_{2n}(\F_\ell)$.  Hence
$\Sp_{2n}(\F_\ell)$ is contained in
$G^{\mathrm{arith}}_\ell$.

On the other hand, the mod-$\ell$ cyclotomic character
$G_\Q\rightarrow\F_\ell^\times$ is surjective.  Since
$\pi_1(\mathcal U)\rightarrow G_\Q$ is also surjective, the multiplier
maps $G^{\mathrm{arith}}_\ell$ onto $\F_\ell^\times$. The 
sequence
\[
1\longrightarrow
\Sp_{2n}(\F_\ell)
\longrightarrow
\GSp_{2n}(\F_\ell)
\longrightarrow
\F_\ell^\times
\longrightarrow1
\]
is exact.  It follows that
$G^{\mathrm{arith}}_\ell=\GSp_{2n}(\F_\ell)$.
\end{proof}

We now make the linear change of variable $T=1+mU$. Since this is an automorphism of the rational function field, the generic torsion extension over $\Q(U)$ still has Galois group $\GSp_{2n}(\F_\ell)$. For an integer $u$, let $C_u$ be given by $y^2=(x-(1+mu))f(x)$, put $J_u=\Jac(C_u)$, and let $L_u=\Q(J_u[\ell])$.

\begin{proposition}\label{prop:HIT}
Let $\mathcal E(Y)$ be the set of integers $1\leq u\leq Y$ for which
$\Gal(L_u/\Q)$ is a proper subgroup of $\GSp_{2n}(\F_\ell)$. Then $\#\mathcal E(Y)\ll_{n,\ell}Y^{1/2}\log Y$.
\end{proposition}

\begin{proof}
By Proposition~\ref{prop:generic-image}, the finite extension of $\Q(U)$ cut out by
the kernel of the generic representation on $\mathcal J[\ell]$ is
Galois with group $\GSp_{2n}(\F_\ell)$.  After deleting the finitely
many branch points of this extension, we obtain a finite etale Galois
cover of a nonempty open subset $\mathcal V\subset\mathbf A^1_\Q$ with
this Galois group.  For every rational point $u\in\mathcal V(\Q)$, the
specialized decomposition group is, up to conjugacy, the image of
$G_\Q$ on $J_u[\ell]$.  In particular, its image is the full generic
group precisely when $\Gal(L_u/\Q)=\GSp_{2n}(\F_\ell)$.

To apply quantitative Hilbert irreducibility, choose a primitive element
for the generic Galois extension and, after clearing denominators, let
$F(X,U)\in\Z[U][X]$ be a polynomial whose splitting field over $\Q(U)$
is the generic torsion field.  Enlarge the finite exceptional set of parameters so
that, away from it, specialization of this polynomial agrees with
specialization of the finite etale cover.  Quantitative Hilbert
irreducibility may then be applied to $F$.  The integral one-parameter
Cohen--Serre estimate, in the form stated in
\cite[Theorem 1.2]{Zywina2010}, gives
\begin{align*}
&\#\{1\leq u\leq Y\mid\Gal(F(X,u)/\Q)
       \text{ is smaller than the generic group}\}\\
&\hspace{4cm}\ll_{F}Y^{1/2}\log Y.
\end{align*}
The polynomial $F$ and hence the implied constant depend only on
$n$ and $\ell$ through the fixed family.  The finitely many parameters
removed in constructing $\mathcal V$ contribute only $O_{n,\ell}(1)$.
This proves the asserted bound.
\end{proof}

\section{Local monodromy at the varying primes}\label{sec:local}

We next determine the ramification outside the fixed set of primes dividing $m$. We shall use throughout the identity $\Disc_x((x-T)f(x))=\Disc(f)f(T)^2$, up to sign. It follows from the product formula for discriminants, since the resultant of $x-T$ and $f(x)$ is $f(T)$.

\begin{lemma}\label{lem:good-reduction}
Let $u\in\Z$ and let $p\nmid m$. If $p\nmid P(u)$, then $C_u$ has good reduction at $p$. Consequently $L_u/\Q$ is unramified at $p$.
\end{lemma}

\begin{proof}
By Lemma~\ref{lem:fixed-primes}, the assumption $p\nmid m$ implies $p\nmid2\ell\Disc(f)$. Thus the fixed branch points $a_1,\ldots,a_{2n}$ remain pairwise distinct modulo $p$, and $p\neq2,\ell$. If $p\nmid P(u)=f(1+mu)$, the moving branch point $1+mu$ is also distinct modulo $p$ from every $a_i$. The reduction of $(x-(1+mu))f(x)$ is therefore squarefree. The odd-degree hyperelliptic model, including its unique point at infinity, is smooth over $\F_p$. Hence $C_u$ and $J_u$ have good reduction at $p$. The N\'eron--Ogg--Shafarevich criterion then implies that the $\ell$-torsion representation is unramified at $p$; see \cite{SerreTate1968}.
\end{proof}

If an
abelian variety $A/\Q_p$ has semistable reduction, then the identity
component of the special fibre of its N\'eron model is an extension of an
abelian variety by a torus.  The dimension of this torus is called the
\emph{toric rank} of $A$ at $p$. 

\begin{proposition}\label{prop:local-transvection}
Let $u\in\Z$ and let $p\nmid m$. Suppose that $p\mid P(u)$ but
$p^2\nmid P(u)$. Then $C_u$ has semistable reduction at $p$, with a
special fibre having a unique ordinary double point.  The corresponding
Jacobian $J_u$ has toric rank one.  Moreover, the image of inertia on
$J_u[\ell]$ is cyclic of order $\ell$, and every nonidentity element of
this image is a symplectic transvection.
\end{proposition}

\begin{proof}
Put $t=1+mu$, so that
$P(u)=f(t)=\prod_{j=1}^{2n}(t-a_j)$.  We first determine the reduction
of the branch points.  Since $p\nmid m=2\ell(2n)!$, we have
$p\neq2,\ell$ and $p>2n$.  The reductions of
$a_1,\ldots,a_{2n}$ modulo $p$ are therefore pairwise distinct.
Indeed, if $a_i\equiv a_j\pmod p$ for $i\neq j$, then
$p\mid a_i-a_j=m(i-j)$.  Since $p\nmid m$, this would imply
$p\mid i-j$, which is impossible because $0<|i-j|<2n<p$.

The condition $p\mid P(u)$ consequently determines a unique index
$i$ for which $t\equiv a_i\pmod p$.  For every $j\neq i$, the
difference $t-a_j$ is a $p$-adic unit.  Hence
$v_p(P(u))=v_p(t-a_i)$, and the assumption $p^2\nmid P(u)$ gives
$v_p(t-a_i)=1$.  Thus precisely two branch points of the hyperelliptic
cover $C_u\rightarrow\PP^1$, namely the moving branch point $t$ and
the fixed branch point $a_i$, have the same reduction modulo $p$.
All of the remaining branch points remain distinct.

We now describe the corresponding singularity.  After translating the
$x$-coordinate we may assume that $a_i=0$, and we write
$\delta=t-a_i$.  Then, near the point where the two branch points meet,
the equation of $C_u$ takes the form
\[
 y^2=x(x-\delta)g(x),
\]
where $v_p(\delta)=1$, $g(0)\in\Z_p^\times$, and the reduction of
$g$ does not vanish at the origin.  Since $p$ is odd, after passing to
an etale neighbourhood the unit $g(x)$ admits a square root and can be
absorbed into the $y$-coordinate.  This is the local
simplification for a hyperelliptic curve with a pair of colliding
branch points; compare
\cite[Propositions~2.2 and~2.3]{AriasDeReynaEtAl2016}.
We are therefore reduced locally to the equation
$y^2=x(x-\delta)$.

Set $X=2x-\delta$ and $Y=2y$.  Then
\[
 (Y-X)(Y+X)=-\delta^2.
\]
Writing $U=Y-X$ and $V=Y+X$, the completed local equation is therefore
$UV=-\delta^2$.  Since $v_p(\delta)=1$, the node has thickness two.
In particular, the natural hyperelliptic model has a special fibre
with exactly one ordinary double point, and there are no other
singularities because every other branch point remains simple modulo
$p$.  This is the local form of a semistable nodal
degeneration; see again
\cite[Propositions~2.2 and~2.3]{AriasDeReynaEtAl2016}.

To determine the toric rank, we describe the special fibre more explicitly.  Modulo
$p$, after translating $a_i$ to zero, its equation has the form
$y^2=x^2\overline g(x)$.  Its normalization is obtained by writing
$z=y/x$, and is given by
$z^2=\overline g(x)$.  The polynomial $\overline g$ is squarefree of
degree $2n-1$, so the normalization is a smooth hyperelliptic curve of
genus $n-1$.  The two points above $x=0$ are identified in the
original special fibre.  Thus the special fibre is geometrically
irreducible and has one nonseparating node.  Its dual graph consists
of one vertex with one loop.

For a semistable curve, the toric part of the identity component of
the N\'eron model of the Jacobian is controlled by the dual graph:
its character group is canonically identified with the first homology
of that graph.  In particular, the toric rank is the first Betti
number of the dual graph; see
\cite[Proposition~1.1]{AriasDeReynaEtAl2016} and
\cite[Corollary~1.4]{Lorenzini1990}.  Since the graph above has one
loop, its first Betti number is one.  It follows that $J_u$ has
semistable reduction of toric rank one.  Equivalently, this follows
from the standard hyperelliptic criterion that one double root and all
remaining roots simple give toric dimension one; compare the
semistable reduction criterion used in
\cite[\S2]{AnniDokchitser2020}.

We shall also need the thickness of the node.  The total space with
local equation $UV=-\delta^2$ is not regular at the singular point.
Resolving a node of thickness two inserts one rational component.
Thus the regular semistable model has two irreducible components, the
strict transform of the original component and one exceptional
rational component, meeting in two points.  The resulting dual graph
has two vertices joined by two edges and again has first Betti number
one.  The same resolution calculation shows that the component group
of the N\'eron model of $J_u$ has order two; see the discussion and
the proof of
\cite[Proposition~2.3]{AriasDeReynaEtAl2016}.  In particular,
$\ell\nmid|\Phi_p|$, since $\ell$ is odd.

We now determine the action of inertia. Unlike the auxiliary prime used
in Proposition~\ref{prop:generic-image}, the prime $p$ is a prime of
semistable reduction, and the inertia action is therefore described by
the monodromy of the node.

For a semistable principally polarized abelian variety at a prime
different from $\ell$, Grothendieck's monodromy description gives,
after choosing a basis adapted to the toric part, an inertia action
whose nontrivial unipotent block is controlled by the integral
monodromy pairing.  In toric rank one this pairing is represented by a
$1\times1$ matrix $N$, and the order of its cokernel is the order of
the component group.  This description is used explicitly in
\cite[Lemma~2.9]{AnniDokchitser2020}; see also
\cite[Lemma~3]{HallOpen2011}.  Since the component group in our case
has order two, the monodromy matrix is $(\pm2)$.  Replacing the
generator of the rank-one lattice by its negative if necessary, we
may write the monodromy matrix as $(2)$.

Thus, after choosing a suitable basis of the $\ell$-adic Tate module,
the action of $\sigma\in I_p$ has the form
\[
 \rho_\ell(\sigma)=1+2t_\ell(\sigma)N_0,
\]
where $t_\ell:I_p\rightarrow\Z_\ell(1)$ is the $\ell$-primary tame
character and $N_0$ is a primitive rank-one nilpotent operator.  Here
$N_0^2=0$.  The coefficient $2$ is precisely the thickness of the
node, or equivalently the value of the monodromy pairing on a
generator of the rank-one toric lattice.

Because $p\neq\ell$, the wild inertia subgroup is pro-$p$ and acts
trivially on this $\ell$-primary unipotent quotient, while the tame
character $t_\ell$ is surjective onto $\Z_\ell(1)$.  Reducing modulo
$\ell$, we therefore obtain
\[
 \rho_\ell(I_p)
 =
 \{1+2c\overline N_0:c\in\F_\ell\}.
\]
Since $\ell$ is odd, the scalar $2$ is nonzero modulo $\ell$, and
$\overline N_0$ is still a nonzero rank-one nilpotent operator.
Consequently the displayed group has exactly $\ell$ elements and is
cyclic.

Every nonidentity element has the form
$1+2c\overline N_0$ with $c\neq0$.  Its difference from the identity
has rank one, and its square is zero.  It is therefore a
transvection.

It remains to see that these transvections are symplectic.  This is the
one point which is the same as in the proof of
Proposition~\ref{prop:generic-image}.  The principal polarization on $J_u$ gives
the Weil pairing, and the multiplier of the Galois action is the
mod-$\ell$ cyclotomic character.  Since $p\neq\ell$, the extension
generated by the $\ell$th roots of unity is unramified at $p$, so the
cyclotomic character is trivial on inertia.  Hence
$\rho_\ell(I_p)\subseteq\Sp_{2n}(\F_\ell)$.  Thus every nonidentity
element of the inertia image is a symplectic transvection.
\end{proof}

\begin{corollary}\label{cor:ram-support-linear}
Suppose that $P(u)$ is squarefree. Outside the set $S$ of primes dividing $m$, the ramification support of $L_u/\Q$ is exactly the set of prime divisors of $P(u)$. At each such prime inertia is cyclic of order $\ell$ and is generated by a transvection.
\end{corollary}

\begin{proof}
If $p\notin S$ and $p\nmid P(u)$, the extension is unramified by Lemma~\ref{lem:good-reduction}. If $p\mid P(u)$, squarefreeness gives $p^2\nmid P(u)$, and Proposition~\ref{prop:local-transvection} applies.
\end{proof}

We shall also need the identical statement for the projective quotient. Let $Z$ be the scalar centre of $\GSp_{2n}(\F_\ell)$. If $L_u/\Q$ has full Galois group, put $\overline L_u=L_u^Z$.

\begin{corollary}\label{cor:ram-support-projective}
Assume that $\Gal(L_u/\Q)=\GSp_{2n}(\F_\ell)$ and that $P(u)$ is squarefree. Then
$\Gal(\overline L_u/\Q)=\PGSp_{2n}(\F_\ell)$. Outside $S$, the ramification support of $\overline L_u/\Q$ is exactly the prime support of $P(u)$, and at every such prime the projective inertia group is cyclic of order $\ell$ generated by the image of a transvection.
\end{corollary}

\begin{proof}
The first assertion follows from Galois theory. Let $p\notin S$ divide $P(u)$. By Corollary~\ref{cor:ram-support-linear}, inertia in $L_u/\Q$ is a group of order $\ell$ generated by a transvection. The centre $Z$ has order $\ell-1$, so this inertia group meets $Z$ trivially. Its image in the projective quotient therefore still has order $\ell$. If $p\notin S$ does not divide $P(u)$, then the linear extension is unramified and so is every quotient. This proves the claim.
\end{proof}

\section{Squarefree values of the degeneration polynomial}\label{sec:sieve}

We now prove that $P(u)$ is squarefree for a positive proportion of integers. No deep squarefree-value theorem is required because $P$ is a product of linear factors and the arithmetic progression separates those factors at every relevant prime.

For $1\leq i\leq2n$, write $L_i(u)=1+m(u-i)$, so $P(u)=\prod_iL_i(u)$.

\begin{lemma}\label{lem:linear-coprime}
For $i\neq j$, the integers $L_i(u)$ and $L_j(u)$ are relatively prime for every $u\in\Z$.
\end{lemma}

\begin{proof}
Suppose that a prime $p$ divides both. Then $p$ divides their difference $m(j-i)$. If $p\mid m$, this is impossible because $L_i(u)\equiv1\pmod p$. Hence $p\nmid m$, and therefore $p\mid(j-i)$. But every prime divisor of $j-i$ is at most $2n$ and hence divides $(2n)!$, so it divides $m$, a contradiction.
\end{proof}

\begin{lemma}\label{lem:local-squarefree-density}
Let $p\nmid m$. There are exactly $2n$ residue classes $u\pmod{p^2}$ for which $p^2\mid P(u)$.
\end{lemma}

\begin{proof}
Since $m$ is invertible modulo $p^2$, each congruence $L_i(u)\equiv0\pmod{p^2}$ has exactly one solution modulo $p^2$. These $2n$ solutions are distinct. Indeed, if the solutions for $i$ and $j$ were equal, then $p^2$ would divide $L_i(u)-L_j(u)=m(j-i)$. Since $p\nmid m$, this would give $p^2\mid(j-i)$. On the other hand, $p\nmid m$ implies $p>2n$, while $0<|i-j|<2n$. This is impossible. By Lemma~\ref{lem:linear-coprime}, $p^2\mid P(u)$ occurs precisely when $p^2$ divides one of the factors, so the listed classes are all the possibilities.
\end{proof}

\begin{proposition}\label{prop:squarefree}
There is a constant $c_{n,\ell}>0$ such that, for all sufficiently large $Y$, the number of integers $1\leq u\leq Y$ for which $P(u)$ is squarefree is at least $c_{n,\ell}Y$.
\end{proposition}

\begin{proof}
Fix a real number $z>2n$.  For every prime $p\leq z$ with $p\nmid m$,
remove the $2n$ residue classes modulo $p^2$ described in
Lemma~\ref{lem:local-squarefree-density}.  Since the moduli $p^2$ are
pairwise coprime, the Chinese remainder theorem shows that the surviving
classes modulo their product have density
\[
 \delta_z=\prod_{\substack{p\leq z\\p\nmid m}}
 \left(1-\frac{2n}{p^2}\right).
\]
Consequently the number of integers $1\leq u\leq Y$ which survive all
these finitely many local conditions is
$\delta_zY+O_{n,\ell,z}(1)$.  Every prime $p\nmid m$ satisfies
$p>2n$, so each Euler factor is positive.  Moreover $\sum_pp^{-2}<\infty$, and hence the infinite product $\delta_\infty=\prod_{p\nmid m}(1-2n/p^2)$ is positive. In particular, for all sufficiently large $z$ one has, say,
$\delta_z>\delta_\infty/2$.

It remains to control square divisors coming from primes larger than
$z$.  Suppose $p>z$ and $p^2\mid P(u)$.  By
Lemma~\ref{lem:linear-coprime}, the prime $p$ divides exactly one of the
linear factors $L_i(u)$, and hence $p^2\mid L_i(u)$ for a unique $i$.
For $1\leq u\leq Y$ we have $|L_i(u)|\leq C_{n,\ell}Y$ once $Y$ is
large, and $L_i(u)$ never vanishes.  Therefore
$p\leq C_{n,\ell}^{1/2}Y^{1/2}$.  For fixed $p$ and $i$, the congruence
$p^2\mid L_i(u)$ determines one residue class modulo $p^2$, so it has at
most $Y/p^2+1$ solutions with $1\leq u\leq Y$.  Summing over $i$ and
over $z<p\ll_{n,\ell}Y^{1/2}$ gives
\[
 \ll_{n,\ell}
 Y\sum_{p>z}\frac1{p^2}
 +\pi(CY^{1/2})
 \ll_{n,\ell}
 \frac{Y}{z}+Y^{1/2}.
\]
The final estimate follows from the elementary bounds
$\sum_{p>z}p^{-2}\ll z^{-1}$ and $\pi(x)\leq x$.

Choose $z$ so large that the coefficient of $Y/z$ in the last bound is
less than $\delta_\infty/8$.  With this $z$ fixed, take $Y$ sufficiently
large that both the $O_{n,\ell,z}(1)$ term in the small-prime sieve and
the $O_{n,\ell}(Y^{1/2})$ term in the large-prime estimate are less than
$\delta_\infty Y/8$.  At least a fixed positive proportion of the
integers $u\leq Y$ then survive every condition $p^2\nmid P(u)$.
These are precisely the parameters for which $P(u)$ is squarefree,
and the proposition follows.
\end{proof}

We next observe that distinct squarefree values of $P(u)$ give different ramification supports.

\begin{lemma}\label{lem:distinct-values}
For $u>2n$, the integer $P(u)$ is positive and $P(u)$ is strictly increasing as a function of $u$. Consequently, if $u_1\neq u_2$ are both larger than $2n$ and $P(u_1),P(u_2)$ are squarefree, then the two values have different prime supports.
\end{lemma}

\begin{proof}
If $u>2n$, each factor $1+m(u-i)$ is positive and strictly increases with $u$. Hence their product is positive and strictly increasing. Two positive squarefree integers have the same set of prime divisors only when they are equal. Thus equality of the prime supports would imply $P(u_1)=P(u_2)$, and strict monotonicity gives $u_1=u_2$.
\end{proof}

Combining the sieve with quantitative Hilbert irreducibility gives the set of parameters to which the counting proposition will be applied.

\begin{proposition}\label{prop:good-parameters}
For all sufficiently large $Y$, there are $\gg_{n,\ell}Y$ integers $u$ with $2n<u\leq Y$ such that $P(u)$ is squarefree and
$\Gal(L_u/\Q)=\GSp_{2n}(\F_\ell)$.
\end{proposition}

\begin{proof}
By Proposition~\ref{prop:squarefree}, the squarefree condition holds for at least $c_{n,\ell}Y+O(1)$ integers $u\leq Y$. By Proposition~\ref{prop:HIT}, only $O_{n,\ell}(Y^{1/2}\log Y)=o(Y)$ of all integers $u\leq Y$ have smaller residual image. Removing these exceptional values, as well as the fixed initial range $u\leq2n$, leaves $\gg_{n,\ell}Y$ parameters satisfying both conditions.
\end{proof}

\section{Proofs of the main theorems}\label{sec:proofs}

We begin with the projective case.

\begin{proof}[Proof of Theorem~\ref{thm:PGSp-main}]
Let $S$ be the set of primes dividing $m$, and let $\pi$ be a faithful transitive permutation representation of $\PGSp_{2n}(\F_\ell)$ in $S_d$. Choose one symplectic transvection $\tau$ and let $\overline\tau$ denote its projective image. By Lemma~\ref{lem:GSp-conjugacy}, the $\PGSp$-conjugacy class of $\overline\tau$ is independent of the choice of $\tau$.

For each of the $\gg Y$ parameters supplied by Proposition~\ref{prop:good-parameters}, the extension $L_u/\Q$ has group $\GSp_{2n}(\F_\ell)$, and hence the fixed field $\overline L_u=L_u^Z$ of the scalar centre has group $\PGSp_{2n}(\F_\ell)$. The polynomial $P(u)$ has degree $2n$, is positive and squarefree, is prime to every element of $S$ by Lemma~\ref{lem:fixed-primes}, and its values on the chosen range are distinct by Lemma~\ref{lem:distinct-values}. Finally, Corollary~\ref{cor:ram-support-projective} shows that outside $S$ the projective extension is ramified precisely at the primes dividing $P(u)$ and that its inertia group at each of these primes is generated by a conjugate of $\overline\tau$.

All hypotheses of Proposition~\ref{prop:counting-principle} are therefore satisfied with $G=\PGSp_{2n}(\F_\ell)$ and $r=2n$. The proposition gives exactly the lower bound asserted in Theorem~\ref{thm:PGSp-main}.
\end{proof}

\begin{proof}[Proof of Theorem~\ref{thm:GSp-main}]
Let $\pi$ be a faithful transitive permutation representation of $\GSp_{2n}(\F_\ell)$ in $S_d$, and let $\tau$ be a transvection. We use the same set $S$ and the same $\gg Y$ parameters from Proposition~\ref{prop:good-parameters}. For such a parameter, $L_u/\Q$ has full Galois group $\GSp_{2n}(\F_\ell)$. By Corollary~\ref{cor:ram-support-linear}, outside $S$ its ramification support is exactly the prime support of $P(u)$ and inertia at every varying prime is generated by a transvection. By Lemma~\ref{lem:GSp-conjugacy}, every such inertia generator is conjugate in $\GSp_{2n}(\F_\ell)$ to the fixed transvection $\tau$.

The remaining hypotheses of Proposition~\ref{prop:counting-principle} are the same as in the projective proof. Applying that proposition with $r=2n$ gives exactly the lower bound asserted in Theorem~\ref{thm:GSp-main}.
\end{proof}

\section{Explicit permutation representations}\label{sec:actions}

We now compute the index appearing in the main theorems for the natural actions. Let $V$ be a $2n$-dimensional symplectic space, and write a transvection in the form
$\tau(x)=x+a\langle x,w\rangle w$ with $w\neq0$ and $a\neq0$. Its fixed subspace is the hyperplane $w^\perp$, which has cardinality $\ell^{2n-1}$.

\begin{proposition}\label{prop:vector-index}
In the action of $\GSp(V)$ on $V\setminus\{0\}$, one has $\ind(\tau)=\ell^{2n-2}(\ell-1)^2$.
\end{proposition}

\begin{proof}
The set $V\setminus\{0\}$ has $\ell^{2n}-1$ elements. The fixed nonzero vectors are the elements of $w^\perp\setminus\{0\}$, so there are $\ell^{2n-1}-1$ fixed points. Hence the number of moving points is $\ell^{2n-1}(\ell-1)$.

The transvection has order $\ell$. Every orbit of the cyclic group $\langle\tau\rangle$ therefore has size one or $\ell$. The moving points consequently break into $\ell^{2n-2}(\ell-1)$ cycles of length $\ell$. The contribution of these cycles to the permutation index is $(\ell-1)$ times their number, giving $\ell^{2n-2}(\ell-1)^2$.
\end{proof}

\begin{proposition}\label{prop:projective-index}
In the natural action of $\PGSp(V)$ on $\PP(V)$, one has $\ind(\overline\tau)=\ell^{2n-2}(\ell-1)$.
\end{proposition}

\begin{proof}
The projective space has $(\ell^{2n}-1)/(\ell-1)$ points. A line is fixed by $\overline\tau$ precisely when it is represented by a fixed vector of $\tau$. Indeed, $\tau$ is unipotent, so if $\tau(v)$ is a scalar multiple of $v$, that scalar must be $1$. The fixed projective points are therefore the lines contained in $w^\perp$, of which there are $(\ell^{2n-1}-1)/(\ell-1)$. Hence exactly $\ell^{2n-1}$ projective points move.

By Lemma~\ref{lem:GSp-conjugacy}, the projective transvection has order $\ell$. Thus the moving points form $\ell^{2n-2}$ cycles, each of length $\ell$, and the index is $\ell^{2n-2}(\ell-1)$.
\end{proof}

The preceding propositions compute the inertia indices which occur in
our construction.  We next compute the least nontrivial permutation
indices in these actions, in order to justify the comparison with
Malle's conjecture made in the introduction.

\begin{proposition}\label{prop:malle-invariants-natural}
Let $\pi_{\mathrm{vec}}$ be the natural action of
$\GSp_{2n}(\F_\ell)$ on $V\setminus\{0\}$, and let
$\pi_{\mathrm{proj}}$ be the natural action of
$\PGSp_{2n}(\F_\ell)$ on $\PP(V)$.  If $n\geq2$, then
\[
 a_{\pi_{\mathrm{vec}}}\bigl(\GSp_{2n}(\F_\ell)\bigr)
 =\frac{\ell^{2n-2}(\ell^2-1)}{2}
\]
and
\[
 a_{\pi_{\mathrm{proj}}}\bigl(\PGSp_{2n}(\F_\ell)\bigr)
 =\frac{(\ell+1)(\ell^{2n-2}-1)}{2}.
\]
For $n=1$, the corresponding values are
$\ell(\ell-1)/2$ and $(\ell-1)/2$, respectively.
\end{proposition}

\begin{proof}
We first record a reduction which will be used in both actions.  If a
nonidentity permutation $g$ has composite order, choose a nontrivial
power $h\in\langle g\rangle$ of prime order.  Every
$\langle g\rangle$-orbit is a union of $\langle h\rangle$-orbits, so
$\ind(h)\leq\ind(g)$.  Thus the minimum defining the Malle invariant is
attained by an element of prime order.  If such an element has prime
order $r$ on a set $\Omega$, then every orbit has size $1$ or $r$, and hence $\ind(g)=\frac{r-1}{r}(|\Omega|-|\operatorname{Fix}_\Omega(g)|)$.

Consider first the vector action and assume $n\geq2$.  Put $q=\ell$.
If $g$ has order $q$, then $g$ is unipotent.  Since $g\neq1$, its fixed
subspace has dimension at most $2n-1$, so $\ind(g)\geq q^{2n-2}(q-1)^2$.
For $q\geq3$ this is at least
$q^{2n-2}(q^2-1)/2$, with equality only when $q=3$ and the fixed
subspace is a hyperplane.

Now suppose that $g$ has prime order $r\neq q$, and write $\mu(g)$ for
its similitude multiplier.  The element is semisimple.  If
$\mu(g)=1$, then the $1$-eigenspace is a nondegenerate symplectic
subspace. The symplectic pairing pairs the $\lambda$- and
$\lambda^{-1}$-eigenspaces, so the $1$-eigenspace is orthogonal to all
other eigenspaces and the restriction of the pairing to it is
nondegenerate.  Since $g$ is nontrivial, this fixed space has dimension at most $2n-2$. Hence $\ind(g)\geq \frac12(q^{2n}-q^{2n-2})=\frac{q^{2n-2}(q^2-1)}2$.
If $\mu(g)\neq1$, then the fixed space is totally isotropic, because
for fixed vectors $v,w$ one has
$\langle v,w\rangle=\mu(g)\langle v,w\rangle$.  Its dimension is therefore at most $n$, and since $n\geq2$ the same lower bound follows from $\ind(g)\geq\frac12(q^{2n}-q^n)\geq\frac12(q^{2n}-q^{2n-2})$.
The bound is attained by a symplectic involution which is $-1$ on a
nondegenerate symplectic plane and $1$ on its symplectic orthogonal
complement.  This proves the vector formula for $n\geq2$.

For $n=1$, a nontrivial semisimple element of multiplier one has no
nonzero fixed vector, whereas an element with nontrivial multiplier can
have a one-dimensional fixed space.  The involution
$\operatorname{diag}(1,-1)\in\GL_2(\F_q)=\GSp_2(\F_q)$ realizes this
maximum and has index $q(q-1)/2$.  Unipotent elements have index
$(q-1)^2$, which is larger for odd $q$.  Hence the stated $n=1$ vector
value follows.

We next consider the projective action.  Again it is enough to consider a
projective element $\overline g$ of prime order $r$.  If it fixes a
projective point, choose a lift $g\in\GSp(V)$ and scale it so that one
fixed vector has eigenvalue $1$.  Since $\overline g^r=1$, this scaling
makes $g^r=1$; thus $g$ itself may be taken to have order $r$.  If
$\overline g$ has no fixed point, the resulting permutation index is
larger than the bounds below, so there is nothing to prove.

If $r=q$, then $g$ is unipotent and has only the eigenvalue $1$.  Its
fixed vector space has dimension at most $2n-1$, so the number of fixed projective points is at most $(q^{2n-1}-1)/(q-1)$. It follows that $\ind(\overline g)\geq q^{2n-2}(q-1)$.
For $n\geq2$ this is at least
$(q+1)(q^{2n-2}-1)/2$.

Suppose next that $r\neq q$.  Then $g$ is semisimple.  A projective
point fixed by $\overline g$ is an $\F_q$-rational eigenline of $g$.
Let $E_\lambda$ denote the eigenspace for an eigenvalue
$\lambda\in\F_q^\times$.  The similitude relation shows that
$E_\lambda$ pairs nondegenerately with $E_{\mu/\lambda}$, where
$\mu$ is the multiplier of $g$.  Consequently, if
$\lambda^2\neq\mu$, the two eigenspaces have the same dimension, while
if $\lambda^2=\mu$, the restriction of the alternating form to
$E_\lambda$ is nondegenerate and hence $\dim E_\lambda$ is even.  In
particular a non-scalar similitude has no eigenspace of dimension
$2n-1$.  Among all collections of rational eigenspaces with total
dimension at most $2n$ and with no eigenspace of dimension $2n-1$, the
number of eigenlines is therefore at most $(q^{2n-2}-1)/(q-1)+(q^2-1)/(q-1)$.
Indeed, the function $d\mapsto(q^d-1)/(q-1)$ is strictly convex in
$d$, so the number of lines is maximized by concentrating the available
dimension into blocks of dimensions $2n-2$ and $2$.  Hence
\[
 \ind(\overline g)
 \geq
 \frac12\left(
 \frac{q^{2n}-1}{q-1}
 -\frac{q^{2n-2}-1}{q-1}
 -\frac{q^2-1}{q-1}\right)
 =\frac{(q+1)(q^{2n-2}-1)}2.
\]
Equality is attained by the projective image of the same symplectic
involution which is $-1$ on a symplectic plane and $1$ on its
orthogonal complement.  This proves the projective formula for
$n\geq2$.

Finally, when $n=1$, a nontrivial projective semisimple element fixes at
most two points of $\PP^1(\F_q)$, and a split involution fixes exactly
two.  Its index is therefore $(q-1)/2$.  A projective unipotent element
fixes one point and has index $q-1$, so the involution gives the
minimum.  This completes the proof.
\end{proof}

\begin{remark}\label{rem:malle-transvection}
For the natural vector action with $n\geq2$, the transvection index is
$\ell^{2n-2}(\ell-1)^2$.  It agrees with the Malle invariant only when
$\ell=3$; for $\ell>3$ the involution in the preceding proof has strictly
smaller index.  For the natural projective action, the involution has
strictly smaller index than a transvection for every odd $q$.
\end{remark}

\begin{proof}[Proof of Corollary~\ref{cor:natural-actions}]
The symplectic group is transitive on the nonzero vectors of $V$, and hence so is $\GSp(V)$. The action is plainly faithful. Similarly $\GSp(V)$ acts transitively on the lines in $V$, and the kernel of this projective action is exactly the scalar centre; the induced action of $\PGSp(V)$ is therefore faithful and transitive. Substituting the indices from Propositions~\ref{prop:vector-index} and~\ref{prop:projective-index} into Theorems~\ref{thm:PGSp-main} and~\ref{thm:GSp-main} gives the two asserted exponents.
\end{proof}

The same argument gives the following bound for the regular action.

\begin{corollary}\label{cor:regular}
Let $G$ be either $\GSp_{2n}(\F_\ell)$ or $\PGSp_{2n}(\F_\ell)$, and let $G$ act regularly on itself. A transvection has permutation index $|G|(\ell-1)/\ell$. The present construction therefore gives the exponent $\ell/(2n(\ell-1)|G|)$ in the regular action.
\end{corollary}

\begin{proof}
An element of order $\ell$ has $|G|/\ell$ cycles, all of length $\ell$, in the regular action. Its index is therefore $|G|-|G|/\ell=|G|(\ell-1)/\ell$. The assertion follows from the corresponding main theorem.
\end{proof}

\begin{remark}\label{rem:n1}
For $n=1$, the vector action of $\GL_2(\F_\ell)$ has degree $\ell^2-1$ and a transvection has index $(\ell-1)^2$. The projective action of $\PGL_2(\F_\ell)$ has degree $\ell+1$ and index $\ell-1$. Thus the exponents become $1/(2(\ell-1)^2)$ and $1/(2(\ell-1))$, respectively.
\end{remark}

\printbibliography

\end{document}